\documentclass[11pt]{article}

\usepackage[a4paper,margin=1.1in]{geometry}
\usepackage{amsmath,amssymb,amsthm,mathtools,mathrsfs}
\usepackage{enumitem}
\usepackage{array}
\usepackage{booktabs}
\usepackage{hyperref}
\hypersetup{
  hidelinks,
  pdftitle={Positivity and Asymptotics for Chenevier's Orthogonal Polynomials},
  pdfauthor={Shisong Xu},
  pdfkeywords={orthogonal polynomials, Schur functions, Verblunsky coefficients, Hausdorff moments, automorphic representations}
}
\usepackage{aliascnt}
\usepackage[nameinlink,capitalise]{cleveref}

\newtheorem{theorem}{Theorem}[section]

\newaliascnt{corollary}{theorem}
\newtheorem{corollary}[corollary]{Corollary}
\aliascntresetthe{corollary}

\newaliascnt{proposition}{theorem}
\newtheorem{proposition}[proposition]{Proposition}
\aliascntresetthe{proposition}

\newaliascnt{lemma}{theorem}
\newtheorem{lemma}[lemma]{Lemma}
\aliascntresetthe{lemma}

\newaliascnt{remark}{theorem}
\newtheorem{remark}[remark]{Remark}
\aliascntresetthe{remark}

\newaliascnt{conjecture}{theorem}
\newtheorem{conjecture}[conjecture]{Conjecture}
\aliascntresetthe{conjecture}

\crefname{theorem}{Theorem}{Theorems}
\Crefname{theorem}{Theorem}{Theorems}
\crefname{proposition}{Proposition}{Propositions}
\Crefname{proposition}{Proposition}{Propositions}
\crefname{lemma}{Lemma}{Lemmas}
\Crefname{lemma}{Lemma}{Lemmas}
\crefname{corollary}{Corollary}{Corollaries}
\Crefname{corollary}{Corollary}{Corollaries}
\crefname{remark}{Remark}{Remarks}
\Crefname{remark}{Remark}{Remarks}
\crefname{conjecture}{Conjecture}{Conjectures}
\Crefname{conjecture}{Conjecture}{Conjectures}
\crefname{section}{Section}{Sections}
\Crefname{section}{Section}{Sections}

\newcommand{\R}{\mathbb R}
\newcommand{\C}{\mathbb C}
\newcommand{\Q}{\mathbb Q}
\newcommand{\T}{\mathbb T}
\newcommand{\one}{\mathbf 1}
\newcommand{\dd}{\,\mathrm d}
\DeclareMathOperator{\Si}{Si}

\title{Positivity and Asymptotics for Chenevier's Orthogonal Polynomials}
\author{Shisong Xu\\[0.35em]
\small Department of Mathematics, Nanjing University\\
\small 22 Hankou Road, Nanjing, Jiangsu 210093, People's Republic of China\\
\small Email: \href{mailto:shsxu@smail.nju.edu.cn}{shsxu@smail.nju.edu.cn}}
\date{}

\begin{document}
\maketitle

\begin{abstract}
We prove the strict positivity conjectured by Chenevier for the critical vectors in the
unconditional part of his automorphic Hermite--Minkowski theorem. The proof establishes
strict negativity of all Verblunsky coefficients of a circle measure associated with the
weight $(\arcsin x)/x$ on $(-1,1)$. A positive-kernel formula and the classical Schur
algorithm give these signs, and a para-orthogonal transformation yields positivity in
every degree. We also compute the positive density representing the negative of the
Schur function as a Hausdorff moment generating function. After rescaling their indices
to $[0,1]$, the normalized critical vectors converge weakly to the arcsine law, while
Chenevier's critical scale is asymptotic to $8\pi/n$. At the critical boundary, a single
nonzero effective integral vector is negative for every admissible test function exactly
in odd degree and in degree zero. Finally, we prove an exact first-variation formula for
exponential perturbations of the Legendre measure and derive its asymptotics for endpoint
cusps. For the perturbation leading to Chenevier's weight, the derivative at the Legendre
measure has a $(\log n)/(\pi^2n^2)$ term and an explicit constant at order $n^{-2}$.
The corresponding nonlinear asymptotic remains conjectural.
\end{abstract}



\medskip
\noindent\textbf{Acknowledgements.}
The author thanks the Professor Kenier Castillo for the bibliographical suggestions and for
pointing out the classical Stieltjes/complete-Bernstein reduction in
\cref{prop:classical-stieltjes}.

\section{Introduction}

The weight $(\arcsin x)/x$ on $(-1,1)$ occurs in the unconditional part of Chenevier's
automorphic generalization of the Hermite--Minkowski theorem
\cite[Section 2.14]{Chenevier2020}. The archimedean quadratic forms in his argument,
associated with the limiting functions
\[
 F_\infty(t)=\frac1{\cosh(t/2)},\qquad
 H_\infty(t)=e^{-t/2}F_\infty(t)=\frac{2e^{-t}}{1+e^{-t}},
\]
lead to the orthogonal polynomials $P_n$ for this weight, normalized by $P_n(1)=1$.
His critical vector $v_n=v_n^{H_\infty}$ is given by
\[
 e^{in\theta}P_n(\cos\theta)=\sum_{k=0}^n v_{n,k}e^{2ik\theta}.
\]
Chenevier computed these vectors for $n\leq25$ and conjectured that all their coordinates
are positive in every degree \cite[Sections 2.14 and 6.4]{Chenevier2020}.

\begin{theorem}[Strict positivity]\label{thm:intro-positivity}
For every $n\geq0$ and $0\leq k\leq n$, one has $v_{n,k}>0$.
\end{theorem}

We prove this by passing to a conjugation-symmetric measure on the unit circle.
An explicit positive kernel shows that its Fourier moments of positive index are negative.
The Schur algorithm then gives strictly negative Verblunsky coefficients. The Szeg\H{o}
recurrence and a para-orthogonal transformation yield the positivity of $v_n$.
We also compute the positive density representing the negative Schur function as a
Hausdorff moment generating function. This gives a stronger moment property than is
needed for the sign argument.

The Stieltjes property of $g(-x)$, where $g$ is the negative Schur function, follows
directly from the standard reciprocal duality between Stieltjes and complete Bernstein
functions \cite[Theorem 7.3]{SchillingSongVondracek2012}; the density is then obtained by
Stieltjes inversion. These are applications of established theory. The input specific to
the positivity problem is the strict sign statement for this circle measure and its
consequence for every critical vector. The explicit density provides an additional
Hausdorff moment property, but is not required for the shorter positivity proof.

Put $L=\log2$ and $\psi=\Gamma'/\Gamma$. The critical matrix, value, and scale are
\begin{equation}\label{eq:intro-critical-data}
 \begin{gathered}
 A_n=\left(L+\psi\left(\frac{1+|i-j|}{2}\right)\right)_{0\leq i,j\leq n},\\
 t_n=v_n^{\mathsf T}A_nv_n,\qquad r(n)=2\pi e^{-t_n}.
 \end{gathered}
\end{equation}
Here $t_n$ denotes Chenevier's $t_n^{H_\infty}$, not his $t(n)=\log r(n)$
\cite[Definition 2.18]{Chenevier2020}. Since $P_n(1)=1$, the coordinates of $v_n$ sum to
one. Together with \cref{thm:intro-positivity}, this makes
$\xi_n=\sum_{k=0}^n v_{n,k}\delta_{k/n}$ a probability measure for every $n\geq1$.

\begin{theorem}[Critical values and limiting distribution]\label{thm:intro-limits}
The sequence $(t_n)$ is strictly increasing. As $n\to\infty$,
\begin{equation}\label{eq:intro-limits}
 t_n=\log n-\log4+o(1),\qquad nr(n)\longrightarrow8\pi,
 \qquad \xi_n\Longrightarrow\frac{\dd x}{\pi\sqrt{x(1-x)}}.
\end{equation}
\end{theorem}

The proof uses the variational characterization of $v_n$ and logarithmic energy on
$[0,1]$. Strict positivity also gives a distinction at the critical boundary $r=r(n)$:
a single effective integral vector is negative for every admissible test function exactly
when $n=0$ or $n$ is odd. In positive even degree, such a vector exists for each test
function separately, but no common vector exists. This reflects the rationality distinction
already noted by Chenevier \cite[Corollary 2.16 and the proof of Proposition 6.5]{Chenevier2020}.
These results concern his quadratic forms and do not change the numerical thresholds in
the automorphic finiteness theorem.

We next study the recurrence coefficients through the deformation
\begin{equation}\label{eq:intro-deformation}
 \begin{gathered}
 B_*(x)=\log\frac{\arcsin x}{x},\qquad B_*(0)=0,\\
 \dd\eta_t(x)=\frac{e^{tB_*(x)}}{\displaystyle\int_{-1}^1 e^{tB_*(u)}\,\dd u}\,\dd x.
 \end{gathered}
\end{equation}
Let $\alpha_j(t)$ be the Verblunsky coefficients of its circle lift, defined in
\cref{sec:positivity}. At $t=0$ this is the Legendre measure, with
$\alpha_{n-1}(0)=-1/(2n+1)$; at $t=1$ it is the measure studied above.

\begin{theorem}[Linear response]\label{thm:intro-response}
There is an explicit real constant $C_{\mathrm{lin}}$ such that
\begin{equation}\label{eq:intro-response}
 \left.\frac{\mathrm d}{\mathrm dt}\alpha_{n-1}(t)\right|_{t=0}
 =\frac{\log n}{\pi^2n^2}+\frac{C_{\mathrm{lin}}}{n^2}+o(n^{-2}).
\end{equation}
The constant is given by \eqref{eq:Clinear}.
\end{theorem}

An exact first-variation formula reduces this assertion to the endpoint behaviour of
$\mathcal LB_*$, where $\mathcal Lq=-((1-x^2)q')'$ is the Legendre operator.
The singular part is $2/(\pi\sqrt{1-x^2})$. A convolution of squared central binomial
coefficients gives the logarithmic term and its constant. The argument applies more
generally to even perturbations with this type of endpoint singularity
(\cref{thm:endpoint-response}). It treats the full perturbation, not a finite Fourier
truncation. The corresponding asymptotic at $t=1$ remains conjectural.

The density itself has a Chebyshev expansion due to Mathar
\cite[Appendix C of the expanded version]{Mathar2006}; the coefficients studied here
instead expand the orthogonal polynomials generated by that density. Our sign argument
uses the classical Schur and Szeg\H{o} theory \cite{Simon2005} and the
Delsarte--Genin transformation \cite{BraccialiSriRangaSwaminathan2016}.
Standard background on real-line orthogonal polynomials may be found
in \cite{Chihara1978,Szego1975}.

\Cref{sec:positivity} derives the moment transforms and proves the sign and density
formulas. \Cref{sec:critical} treats the critical values, limiting vectors, and integral
boundary obstruction. \Cref{sec:response} proves the linear-response theorem and states
the remaining nonlinear problem. Appendix~\ref{app:response} contains the individual Fourier-mode
expansions and endpoint sums. Throughout, $\gamma$ is Euler's constant,
$H_m=\sum_{j=1}^m j^{-1}$ with $H_0=0$, and
$\Si(x)=\int_0^x(\sin t)/t\,\dd t$.
The symbol $\one$ denotes the vector with every coordinate equal to one.

\section{Orthogonal polynomials and strict positivity}\label{sec:positivity}

Let $\mu$ be the finite even measure on $[-1,1]$ defined by
\begin{equation}\label{eq:mu}
  \dd\mu(x)=\frac{1}{\pi}\frac{\arcsin x}{x}\,\dd x,
\end{equation}
where the density is continuously extended at $x=0$, and let
$\widehat\mu=L^{-1}\mu$ be the corresponding probability measure. Let $(P_n)_{n\geq0}$
be the orthogonal polynomial sequence for $\widehat\mu$, normalized by $P_n(1)=1$.

Since $\widehat\mu$ is even and has support $[-1,1]$,
\begin{equation}\label{eq:rw-recurrence}
  P_0(x)=1,\qquad P_1(x)=x,
\end{equation}
\begin{equation}\label{eq:rw-recurrence2}
  xP_n(x)=(1-c_n)P_{n+1}(x)+c_nP_{n-1}(x),
  \qquad n\geq1,
\end{equation}
where $0<c_n<1$ and $c_0=0$. The monic polynomials $\pi_n$, with $\pi_0=1$ and $\pi_1=x$, satisfy
\begin{equation}\label{eq:monic-recurrence}
  \pi_{n+1}(x)=x\pi_n(x)-\beta_n\pi_{n-1}(x),
  \qquad
  \beta_n=(1-c_{n-1})c_n,\qquad n\geq1.
\end{equation}
This is the recurrence of a symmetric random-walk polynomial sequence
\cite{CharrisIsmail1986,vanDoornSchrijner1993}.

Write the Chebyshev expansion
\begin{equation}\label{eq:cheb-expansion}
  P_n(x)=\sum_{j=0}^{\lfloor n/2\rfloor}r_n(j)T_{n-2j}(x),
\end{equation}
where $T_m(\cos\theta)=\cos(m\theta)$. Equivalently,
\begin{equation}\label{eq:laurent-vector}
  e^{in\theta}P_n(\cos\theta)
  =\sum_{k=0}^n v_{n,k}e^{2ik\theta}.
\end{equation}
The two coefficient systems satisfy
\begin{equation}\label{eq:r-v-relation}
  v_{n,k}=\frac12r_n(k)\quad(0\leq k<n/2),
  \qquad v_{n,n-k}=v_{n,k},
\end{equation}
and, for even $n$,
\begin{equation}\label{eq:r-v-middle}
  v_{n,n/2}=r_n(n/2).
\end{equation}
Under the change of angular variable used in \cite[Proposition 2.9 and Section 2.14]{Chenevier2020},
$(v_{n,0},\ldots,v_{n,n})$ is precisely the critical vector $v_n=v_n^{H_\infty}$.

For $n\geq0$, set
\[
  s_n=\int_{-1}^1x^{2n}\,\dd\mu(x).
\]
All odd moments vanish.

The following moment formula is due to Chenevier \cite[Section 2.14]{Chenevier2020}.
We derive it together with the two transforms used below.

\begin{proposition}\label{prop:moments-transform}
One has
\begin{equation}\label{eq:moment-formula}
  s_0=L,
  \qquad
  s_n=\frac{1}{2n}\left(1-\frac{\binom{2n}{n}}{4^n}\right)
  \quad(n\geq1).
\end{equation}
Consequently,
\begin{equation}\label{eq:moment-generating}
  \sum_{n=0}^{\infty}s_nz^n
  =\log\left(1+\frac{1}{\sqrt{1-z}}\right),
  \qquad |z|<1,
\end{equation}
and the Cauchy--Stieltjes transform is
\begin{equation}\label{eq:cauchy-transform}
  \mathcal C_\mu(\zeta)
  :=\int_{-1}^1\frac{\dd\mu(x)}{\zeta-x}
  =\frac{1}{\zeta}
  \log\left(1+\frac{\zeta}{\sqrt{\zeta^2-1}}\right),
\end{equation}
for $\zeta\in\C\setminus[-1,1]$, where
$\sqrt{\zeta^2-1}\sim\zeta$ at infinity and the logarithm is the analytic branch whose
value tends to $\log 2$ as $\zeta\to\infty$.
\end{proposition}

\begin{proof}
The total mass is
\[
  s_0=\frac2\pi\int_0^1\frac{\arcsin x}{x}\,\dd x
  =\frac2\pi\int_0^{\pi/2}\theta\cot\theta\,\dd\theta=L,
\]
where the last equality follows by integration by parts and
$\int_0^{\pi/2}\log(\sin\theta)\,\dd\theta=-(\pi/2)\log 2$.
For $n\geq1$, integration by parts gives
\begin{align*}
  s_n
  &=\frac2\pi\int_0^1x^{2n-1}\arcsin x\,\dd x\\
  &=\frac{1}{2n}-\frac{1}{n\pi}
    \int_0^1\frac{x^{2n}}{\sqrt{1-x^2}}\,\dd x\\
  &=\frac{1}{2n}\left(1-\frac{\binom{2n}{n}}{4^n}\right).
\end{align*}
Using
\[
  \sum_{n\geq1}\frac{z^n}{2n}=-\frac12\log(1-z),
  \qquad
  \sum_{n\geq1}\frac1n\frac{\binom{2n}{n}}{4^n}z^n
  =2\log\frac{2}{1+\sqrt{1-z}},
\]
we obtain \eqref{eq:moment-generating}. For $|\zeta|>1$,
\[
  \mathcal C_\mu(\zeta)
  =\frac1\zeta\sum_{n\geq0}\frac{s_n}{\zeta^{2n}},
\]
and \eqref{eq:cauchy-transform} follows by analytic continuation.
\end{proof}

Let $\nu$ be the pushforward of $\widehat\mu$ under $x\mapsto y=x^2$. Its Markov
function is
\begin{equation}\label{eq:Mnu}
  M_\nu(z)=\int_0^1\frac{\dd\nu(y)}{1-zy}
  =\frac1L\log\left(1+\frac1{\sqrt{1-z}}\right).
\end{equation}
Push $\nu$ forward under $u=2y-1$ to a probability measure $\lambda$ on $[-1,1]$,
and let $\sigma$ be its conjugation-symmetric Szeg\H{o} lift to $\T$, so that
$\operatorname{Re}z$ has distribution $\lambda$ under $\sigma$.
Equivalently, if $X$ has distribution $\widehat\mu$, then
$e^{2i\arccos X}$ has distribution $\sigma$. We use the same construction for every
even probability measure on $[-1,1]$ below.

\begin{proposition}\label{prop:circle-lift}
The measure $\sigma$ has density
\begin{equation}\label{eq:circle-density}
  \dd\sigma(e^{i\theta})
  =\frac{(\pi-|\theta|)\tan(|\theta|/2)}{4\pi L}\,\dd\theta,
  \qquad -\pi<\theta<\pi,
\end{equation}
with the limiting values at $\theta=0,\pm\pi$. The Carath\'eodory function
$\mathcal F(z)=\int_\T (w+z)/(w-z)\,\dd\sigma(w)$ and the associated Schur function are
\begin{equation}\label{eq:F-Caratheodory}
  \mathcal F(z)=\frac{1-z}{1+z}\frac{L-\log(1-z)}{L},
  \qquad |z|<1,
\end{equation}
\begin{equation}\label{eq:f-Schur}
  f(z)=\frac{\mathcal F(z)-1}{z(\mathcal F(z)+1)}
  =-\frac{2Lz+(1-z)\log(1-z)}
  {z\bigl(2L-(1-z)\log(1-z)\bigr)}.
\end{equation}
\end{proposition}

\begin{proof}
For $0<\theta<\pi$, the substitutions $y=x^2$ and $u=2y-1=\cos\theta$ give
\[
  \arcsin\sqrt{\frac{1+\cos\theta}{2}}=\frac{\pi-\theta}{2}.
\]
The Jacobians yield \eqref{eq:circle-density}. For a symmetric circle measure,
\[
  \mathcal F(z)=\int_{-1}^1\frac{1-z^2}{1-2zu+z^2}\,\dd\lambda(u).
\]
Writing $u=2y-1$ and $q=4z/(1+z)^2$, we obtain
\[
  \mathcal F(z)=\frac{1-z}{1+z}M_\nu(q).
\]
For $z$ near zero, $\sqrt{1-q}=(1-z)/(1+z)$, so \eqref{eq:Mnu} gives
\eqref{eq:F-Caratheodory}. Analytic continuation extends the identity to the unit disk,
and simplification gives \eqref{eq:f-Schur}.
\end{proof}

To determine the signs of the Schur parameters, consider the positive kernel
\begin{equation}\label{eq:Dstar}
 D_*(z)=\int_0^1\frac{1-s}{1+s}\frac{\dd s}{1-sz},\qquad z\in\C\setminus[1,\infty).
\end{equation}
Partial fractions give
\[
 z(1+z)D_*(z)=2Lz+(1-z)\log(1-z).
\]
Together with \eqref{eq:F-Caratheodory}, this yields, for $|z|<1$,
\begin{equation}\label{eq:Dstar-F}
 \mathcal F(z)=1-\frac{zD_*(z)}L,\qquad
 g(z):=-f(z)=\frac{D_*(z)}{2L-zD_*(z)}.
\end{equation}
In particular, the Fourier moments $m_k^\sigma=\int_\T z^{-k}\,\dd\sigma(z)$ satisfy
\begin{equation}\label{eq:negative-circle-moments}
 m_k^\sigma=-\frac1{2L}\int_0^1 s^{k-1}\frac{1-s}{1+s}\,\dd s<0,
 \qquad k\geq1.
\end{equation}
The Schur parameters are defined by $f_0=f$ and
\begin{equation}\label{eq:Schur-iteration}
 \alpha_j=f_j(0),\qquad
 f_{j+1}(z)=\frac{f_j(z)-\alpha_j}{z(1-\overline{\alpha_j}f_j(z))}.
\end{equation}
They are the Verblunsky coefficients of $\sigma$ \cite[Chapter 1]{Simon2005}.

\begin{lemma}\label{lem:negative-preserved}
If a Schur function has the expansion $h(z)=-\sum_{m\geq0}d_mz^m$ with every $d_m>0$,
then every Taylor coefficient of its next Schur iterate is also strictly negative.
\end{lemma}
\begin{proof}
Put $G=-h$ and $a=d_0$. Since $h$ is nonconstant, $0<a<1$. Its next iterate is
\[
 h_1(z)=-\frac{(G(z)-a)/z}{1-aG(z)},\qquad
 \frac1{1-aG(z)}=\frac1{1-a^2}
 \sum_{r\geq0}\left(\frac{a(G(z)-a)}{1-a^2}\right)^r.
\]
The series $(G-a)/z$ has strictly positive coefficients. Expanded as a formal power series, the reciprocal denominator has
nonnegative coefficients and a positive constant term. Hence every coefficient of $h_1$
is strictly negative.
\end{proof}

\begin{theorem}[Strict Schur signs]\label{thm:alpha-negative}
Every Taylor coefficient of every Schur iterate of $f$ is strictly negative. In particular,
\begin{equation}\label{eq:alpha-negative}
 -1<\alpha_j<0\qquad(j\geq0).
\end{equation}
\end{theorem}
\begin{proof}
Every Taylor coefficient of $D_*$ is strictly positive by \eqref{eq:Dstar}. Expanding
\eqref{eq:Dstar-F} at the origin gives
\[
 g(z)=\frac{D_*(z)}{2L}\sum_{r\geq0}
        \left(\frac{zD_*(z)}{2L}\right)^r.
\]
Each Taylor coefficient is strictly positive. Thus $f$ has strictly negative coefficients,
and \cref{lem:negative-preserved} applies at every step. Each iterate is nonconstant,
so $|\alpha_j|<1$.
\end{proof}

\begin{remark}
The same argument applies to any conjugation-symmetric circle probability measure of
infinite support whose Fourier moments of positive index are strictly negative.
Writing $\mathcal F=1-2A$, one has $-f=A/(z(1-A))$; its coefficients are strictly
positive because those of $A$ of positive index are strictly positive.
Equation \eqref{eq:negative-circle-moments} verifies this hypothesis for $\sigma$.
For the relation between the Schur algorithm and moment generating functions, see
\cite{Wall1940,Wall1948}.
\end{remark}

Let $\Phi_j$ be the monic orthogonal polynomials for $\sigma$, and put
$\Phi_j^*(z)=z^j\overline{\Phi_j(1/\bar z)}$. Their coefficients are real and
\begin{equation}\label{eq:Szego-recursion}
 \Phi_{j+1}(z)=z\Phi_j(z)-\alpha_j\Phi_j^*(z),\qquad \Phi_0=1.
\end{equation}
The Delsarte--Genin transformation takes the following form with our endpoint
normalization; see \cite{BraccialiSriRangaSwaminathan2016}.

\begin{proposition}[The para-orthogonal formula]\label{prop:para}
For $n\geq1$ and $z=e^{2i\theta}$,
\begin{equation}\label{eq:para}
 e^{in\theta}P_n(\cos\theta)
 =\frac{z\Phi_{n-1}(z)+\Phi_{n-1}^*(z)}{2\Phi_{n-1}(1)}.
\end{equation}
The same identity holds for the circle lift of any even probability measure on $[-1,1]$
with infinite support.
\end{proposition}
\begin{proof}
Set $\mathcal V_n=z\Phi_{n-1}+\Phi_{n-1}^*$. This is a real self-reciprocal polynomial of degree
$n$, with extreme coefficients equal to one. Thus
$e^{-in\theta}\mathcal V_n(e^{2i\theta})$ is a polynomial of degree $n$ in $\cos\theta$, of parity
$n$. For $1\leq\ell\leq n-1$, circle orthogonality gives
\[
 \int_\T z^{-\ell}\mathcal V_n(z)\,\dd\sigma(z)=0.
\]
For the first summand, use the test monomial $z^{\ell-1}$ against $\Phi_{n-1}$.
For the second, use $\Phi_{n-1}^*(z)=z^{n-1}\overline{\Phi_{n-1}(z)}$ on $\T$ and
conjugate the orthogonality relation with $z^{n-1-\ell}$.
For $0\leq r<n$ with $r\equiv n\pmod2$, multiplication by $\cos(r\theta)$ pairs the
indices $\ell=(n-r)/2$ and $(n+r)/2$, both in this range. The corresponding real-line
polynomial is therefore orthogonal to every lower-degree polynomial of parity $n$.
Opposite parity is automatic. Its value at $x=1$ is $2\Phi_{n-1}(1)$, proving the formula.
The denominator is positive: \eqref{eq:Szego-recursion} gives
$\Phi_{n-1}(1)=\prod_{j=0}^{n-2}(1-\alpha_j)>0$ for real $|\alpha_j|<1$.
\end{proof}

\begin{theorem}\label{thm:critical-vector-positive}
For every $n\geq0$, all coefficients $v_{n,k}$ and all Chebyshev coefficients $r_n(j)$
in \eqref{eq:laurent-vector} and \eqref{eq:cheb-expansion} are strictly positive.
For $n\geq1$,
\begin{equation}\label{eq:extreme-coordinates}
 v_{n,0}=v_{n,n}=\frac1{2\prod_{j=0}^{n-2}(1-\alpha_j)}.
\end{equation}
\end{theorem}
\begin{proof}
By \cref{thm:alpha-negative}, each step of \eqref{eq:Szego-recursion} adds a positive
multiple of $\Phi_j^*$ to $z\Phi_j$. Starting from $\Phi_0=1$, induction shows that
every coefficient of every $\Phi_j$ is strictly positive. Formula \eqref{eq:para} proves
the assertion for $v_n$; the case $n=0$ is immediate. The identities
\eqref{eq:r-v-relation}--\eqref{eq:r-v-middle} give the Chebyshev assertion. The constant and leading coefficients of $\mathcal V_n$ are one, giving
\eqref{eq:extreme-coordinates}.
\end{proof}

This proves \cref{thm:intro-positivity}. The leading coefficient of $P_n$, denoted by
$\ell_n$, is $2^{n-1}/\Phi_{n-1}(1)$ for $n\geq1$. Comparison of leading coefficients
in \eqref{eq:rw-recurrence2} therefore gives the scalar Geronimus relation
\begin{equation}\label{eq:alpha-c}
 1-c_n=\frac{\ell_n}{\ell_{n+1}}=\frac{1-\alpha_{n-1}}2,
 \qquad \alpha_{n-1}=2c_n-1\quad(n\geq1).
\end{equation}
In particular,
\begin{equation}\label{eq:c-half}
 0<c_n<\frac12\qquad(n\geq1).
\end{equation}
The argument for \eqref{eq:alpha-c}, unlike the strict inequalities, applies to every
circle lift in \cref{prop:para}.

\label{par:canonical-moments}
The continued-fraction interpretation identifies these $c_n$ as the canonical moments of
$\nu$. Indeed, symmetry and \eqref{eq:monic-recurrence} give
\[
 M_\nu(z)=\cfrac1{1-\cfrac{\zeta_1z}{1-\cfrac{\zeta_2z}{1-\ddots}}},\qquad
 \zeta_1=c_1,\quad \zeta_n=(1-c_{n-1})c_n\quad(n\geq2).
\]
Wall's parametrization of Hausdorff moment sequences has this form
\cite[Theorems 4.1 and 6.1]{Wall1940}; see also \cite{Wall1948,Sokal2020}.
Its parameters are the canonical moments \cite[Chapters 1 and 3]{DetteStudden1997},
so uniqueness of the S-fraction identifies them with the $c_n$.
The affine change $y\mapsto2y-1$ preserves canonical moments, in agreement with
\cite[Theorem 4.3]{DetteWagener2010}. The inequalities \eqref{eq:c-half} also imply
strict positivity of the Chebyshev connection coefficients through Kahler's recursion
\cite[Theorem 2.1]{Kahler2023}.

The Chebyshev coefficients of the density should be distinguished from those of $P_n$.
Mathar's expansion is
\begin{equation}\label{eq:Mathar-weight-expansion}
  \frac{\arcsin x}{x}
  =\frac{a_0}{2}+\sum_{m\geq1}a_{2m}T_{2m}(x),
  \qquad
  a_0=\frac{8G_{\mathrm{Cat}}}{\pi},
  \qquad
  a_{n+2}=-a_n+\frac{8}{\pi(n+1)^2},
\end{equation}
where $n=0,2,4,\ldots$ and $G_{\mathrm{Cat}}$ is Catalan's constant
\cite[Appendix C]{Mathar2006}. These $a_{2m}$ expand the density, whereas the $r_n(j)$ in \eqref{eq:cheb-expansion}
expand its orthogonal polynomials.

We now strengthen the coefficient signs to a Hausdorff moment representation for $g=-f$.
The explicit Schur function gives
\begin{equation}\label{eq:g-def}
 g(z)=\frac{2Lz+(1-z)\log(1-z)}
 {z\bigl(2L-(1-z)\log(1-z)\bigr)}.
\end{equation}
We use the principal branch of $\log(1-z)$ on $\C\setminus[1,\infty)$.
The apparent singularities at $z=0$ and $z=-1$ are removable, with
\[
 g(0)=\frac{2L-1}{2L},\qquad g(-1)=\frac{1-L}{1+L}.
\]
A Stieltjes function on $(0,\infty)$ has the form
\[
 \frac ax+b+\int_{(0,\infty)}\frac{\tau(\dd t)}{x+t},\qquad a,b\geq0,
 \qquad \int_{(0,\infty)}\frac{\tau(\dd t)}{1+t}<\infty.
\]
A nonzero function is Stieltjes if and only if its reciprocal is a complete Bernstein
function. We use this duality and closure of complete Bernstein functions under positive
sums \cite[Theorems 6.2 and 7.3]{SchillingSongVondracek2012}.

\begin{proposition}[Classical Stieltjes reduction]\label{prop:classical-stieltjes}
The function $q(x)=g(-x)$, for $x>0$, is a Stieltjes function.
\end{proposition}
\begin{proof}
Set
\[
 \mathcal H(x)=\frac{(1+x)\log(1+x)}x,\qquad
 \mathcal D(x)=\frac{\mathcal H(x)-2L}{x-1},
\]
where the value at $x=1$ is defined by continuity. Partial fractions in
\eqref{eq:Dstar} identify $\mathcal D(x)=D_*(-x)$. The substitution $s=1/t$ gives
\[
 \mathcal D(x)=\int_1^\infty\frac{t-1}{t(t+1)}\frac{\dd t}{x+t},
\]
so $\mathcal D$ is a nonzero Stieltjes function. Equation \eqref{eq:Dstar-F} now gives
\[
 q(x)=\frac1{x+2L/\mathcal D(x)}.
\]
The denominator is a nonzero complete Bernstein function by reciprocal duality and
closure under positive sums; its reciprocal is Stieltjes.
\end{proof}

This reduction uses only classical reciprocal duality. In the next theorem, Stieltjes
inversion computes the representing measure for the present weight; neither the duality
nor the inversion formula is asserted as a new general result.

\begin{theorem}[Explicit Hausdorff density]\label{thm:stieltjes-schur}
For $z\in\C\setminus[1,\infty)$,
\begin{equation}\label{eq:g-Stieltjes}
  g(z)=\int_0^1\frac{\rho(s)}{1-sz}\,\dd s,
\end{equation}
where
\begin{equation}\label{eq:rho}
  \rho(s)=
  \frac{2L(1-s^2)}
  {\left(2Ls+(1-s)\log\frac{1-s}{s}\right)^2
   +\pi^2(1-s)^2},
  \qquad 0<s<1.
\end{equation}
In particular,
\begin{equation}\label{eq:f-negative-coeff}
  f(z)=-\sum_{m=0}^{\infty}b_mz^m,
  \qquad
  b_m=\int_0^1s^m\rho(s)\,\dd s>0.
\end{equation}
More strongly, for all $k,m\geq0$,
\begin{equation}\label{eq:strict-Hausdorff}
  \sum_{j=0}^k(-1)^j\binom{k}{j}b_{m+j}
  =\int_0^1s^m(1-s)^k\rho(s)\,\dd s>0.
\end{equation}
\end{theorem}

\begin{proof}
By \cref{prop:classical-stieltjes}, $q(x)=g(-x)$ has a Stieltjes representation.
The function $g$ is bounded at zero, and
\[
  \lim_{z\to1,\ z\notin[1,\infty)}g(z)=1,
  \qquad q(x)=O(x^{-1})\quad(x\to\infty).
\]
It follows that the Stieltjes representation of $q$ has neither an $a/x$ term nor a constant
term. Since $q$ extends analytically through $(-1,0)$, its representing measure is supported
on $[1,\infty)$. Thus
$q(x)=\int_{[1,\infty)}(x+t)^{-1}\tau(\dd t)$ for $x>0$ and a positive measure $\tau$.
This measure is finite: monotone convergence applied to
$xq(x)=\int x(x+t)^{-1}\tau(\dd t)$, together with $q(x)=O(x^{-1})$, gives
$\tau([1,\infty))<\infty$.

For $t>1$, use the principal branch of $\log(1-z)$ on
$\C\setminus[1,\infty)$. The upper boundary value satisfies
$\log(1-t-i0)=\log(t-1)-i\pi$, and direct simplification gives
\begin{equation}\label{eq:Im-g}
  \varpi(t):=\frac1\pi\operatorname{Im}g(t+i0)
  =\frac{2L(t^2-1)}
  {t\left(\bigl(2L+(t-1)\log(t-1)\bigr)^2
  +\pi^2(t-1)^2\right)}>0.
\end{equation}
On every compact subinterval of $(1,\infty)$, these boundary values are continuous and are
approached uniformly from the upper half-plane. In particular, for each $u>1$,
\[
  \tau(\{u\})
  =\lim_{\varepsilon\downarrow0}
    \varepsilon\operatorname{Im}g(u+i\varepsilon)=0.
\]
Hence the endpoints of every interval $[a,b]\subset(1,\infty)$ are continuity points of
$\tau$. The Stieltjes inversion formula and dominated convergence now give
\begin{equation}\label{eq:tau-inversion}
  \tau((a,b))
  =\lim_{\varepsilon\downarrow0}\frac1\pi
    \int_a^b\operatorname{Im}g(t+i\varepsilon)\,\dd t
  =\int_a^b\varpi(t)\,\dd t,
  \qquad 1<a<b<\infty.
\end{equation}
Since such intervals form a determining class and exhaust $(1,\infty)$,
\eqref{eq:tau-inversion} shows that
$\tau|_{(1,\infty)}=\varpi(t)\,\dd t$; in particular, there is no additional singular
continuous component. There is also no atom at $t=1$, because $g$ is bounded as $z\to1$
off the cut. Thus the representing measure has been identified completely, and
\[
  g(z)=\frac1\pi\int_1^\infty
  \frac{\operatorname{Im}g(t+i0)}{t-z}\,\dd t
  =\int_1^\infty\frac{\varpi(t)}{t-z}\,\dd t.
\]
The change of variables $s=1/t$ gives \eqref{eq:g-Stieltjes}--\eqref{eq:rho}. Notice that
$\rho(s)=O(\log^{-2}(1/s))$ as $s\downarrow0$ and $\rho(s)=O(1-s)$ as $s\uparrow1$,
so the representing density is integrable. Expanding $(1-sz)^{-1}$ for $|z|<1$ and using
dominated convergence proves \eqref{eq:f-negative-coeff}. The same calculation with
$s^m(1-s)^k$ proves \eqref{eq:strict-Hausdorff}.
\end{proof}

\label{par:Pick-criterion}
The representation \eqref{eq:g-Stieltjes} is also an instance of the Pick-function
characterization of Hausdorff moment generating functions
\cite[Theorem 1]{LiuPego2016}; positivity of $\rho$ gives the strict inequalities.
The corrigendum \cite{LiuPego2026} concerns the proof of Theorem 5 there, not this criterion.
The moments $b_m$ of $\rho(s)\,\dd s$ are distinct from the moments of $\nu$,
whose S-fraction determines the recurrence parameters $c_n$.

The moment formula also gives an exact arithmetic description of the Hankel determinants.
Set
\begin{equation}\label{eq:E-O}
  E_m=\det(s_{i+j})_{0\leq i,j<m},
  \qquad
  O_m=\det(s_{i+j+1})_{0\leq i,j<m},
\end{equation}
with $E_0=O_0=1$.

\begin{proposition}\label{prop:hankel-factorization}
For $N\geq1$, let
\[
  D_N=\det\left(\int_{-1}^1x^{i+j}\,\dd\mu(x)\right)_{0\leq i,j<N}.
\]
Set $D_0=1$. Then, for $m\geq0$,
\begin{equation}\label{eq:D-even-odd}
  D_{2m}=E_mO_m,
  \qquad
  D_{2m+1}=E_{m+1}O_m.
\end{equation}
Moreover,
\begin{equation}\label{eq:beta-even-odd}
  \beta_{2m}=\frac{E_{m+1}O_{m-1}}{E_mO_m},
  \qquad
  \beta_{2m+1}=\frac{E_mO_{m+1}}{E_{m+1}O_m},
\end{equation}
where the first formula holds for $m\geq1$ and the second for $m\geq0$. Finally,
\begin{equation}\label{eq:E-Qlog}
  E_m\in\Q+\Q L,
  \qquad
  O_m\in\Q,
\end{equation}
and hence
\begin{equation}\label{eq:beta-pair-product}
  \beta_{2m}\beta_{2m+1}
  =\frac{O_{m-1}O_{m+1}}{O_m^2}\in\Q_{>0},
  \qquad m\geq1.
\end{equation}
\end{proposition}

\begin{proof}
Since all odd moments vanish, a simultaneous permutation of rows and columns separates the
full Hankel matrix into its even and odd blocks, proving \eqref{eq:D-even-odd}. The standard
formula $\beta_n=D_{n+1}D_{n-1}/D_n^2$ then gives \eqref{eq:beta-even-odd}; see, for
example, \cite{Chihara1978}. By \cref{prop:moments-transform}, $s_0=L$ and
$s_n\in\Q$ for $n\geq1$. The only occurrence of $s_0$ in the matrix defining $E_m$
is its upper-left entry, whereas every entry defining $O_m$ is rational. This proves
\eqref{eq:E-Qlog}, and \eqref{eq:beta-pair-product} follows immediately.
\end{proof}

For example,
\[
  \beta_1=\frac{1}{4L},\qquad
  \beta_2=\frac{5L-2}{8L},\qquad
  \beta_3=\frac{13L}{24(5L-2)},\qquad
  \beta_2\beta_3=\frac{13}{192}.
\]
These identities permit exact computation in $\Q(L)$ using the two parity blocks.

\section{Critical values and limiting vectors}\label{sec:critical}

Let $A_n,t_n,r(n)$ be as in \eqref{eq:intro-critical-data}, and put
\[
 q_n(x)=x^{\mathsf T}A_nx,\qquad \phi_n(x)=\sum_{j=0}^n x_j.
\]
By \cite[Propositions 2.2 and 2.6]{Chenevier2020}, $q_n$ is negative definite on
$\ker\phi_n$, and
\begin{equation}\label{eq:critical-equations}
 A_nv_n=t_n\one,\qquad \phi_n(v_n)=1.
\end{equation}
Consequently,
\begin{equation}\label{eq:variational}
 q_n(x)=t_n+q_n(x-v_n)\leq t_n\qquad(\phi_n(x)=1),
\end{equation}
with equality exactly when $x=v_n$. Thus $v_n$ is the unique maximizer on this affine
hyperplane and, by strict positivity, also on the probability simplex.

\begin{proposition}\label{prop:monotone}
For every $n\geq0$, $t_{n+1}>t_n$ and $r(n+1)<r(n)$.
\end{proposition}
\begin{proof}
The padded vector $(v_n,0)$ has coordinate sum one and $q_{n+1}$-value $t_n$.
It differs from $v_{n+1}$, whose last coordinate is strictly positive.
The uniqueness in \eqref{eq:variational} gives the first assertion, and the second follows
from $r(n)=2\pi e^{-t_n}$.
\end{proof}

The critical matrix is invertible in every degree.
\begin{lemma}\label{lem:An-invertible}
The matrix $A_n$ is invertible for every $n\geq0$, and
\begin{equation}\label{eq:critical-linear-algebra}
 t_n=\frac1{\one^{\mathsf T}A_n^{-1}\one},\qquad
 v_n=\frac{A_n^{-1}\one}{\one^{\mathsf T}A_n^{-1}\one}.
\end{equation}
\end{lemma}
\begin{proof}
The exact values in \cite[Section 2.14, Table 2]{Chenevier2020} are
\[
 t_0=-\gamma-L,\quad t_1=-\gamma,\quad
 t_2=-\gamma+\frac{L}{4L-1},\quad t_3=-\gamma+\frac23.
\]
Since $L>1/2$ and $1/2<\gamma<2/3$, one has
$L/(4L-1)<1/2<\gamma$. Hence $t_0,t_1,t_2<0<t_3$;
\cref{prop:monotone} shows that no $t_n$ vanishes.
If $0\ne w\in\ker A_n$, then \eqref{eq:critical-equations} gives
$0=w^{\mathsf T}A_nv_n=t_n\phi_n(w)$, so $w\in\ker\phi_n$.
Negative definiteness there contradicts $q_n(w)=0$. Thus $A_n$ is invertible,
and \eqref{eq:critical-linear-algebra} follows from \eqref{eq:critical-equations}.
\end{proof}

To study the limiting distribution of the coordinates, let $\xi$ be a probability measure
on $[0,1]$ and define its logarithmic energy by
\[
 I(\xi)=\iint_{[0,1]^2}\log|x-y|\,\dd\xi(x)\,\dd\xi(y)\in[-\infty,0].
\]
Its extremal property has the following elementary proof.

\begin{lemma}\label{lem:energy}
One has $I(\xi)\leq-\log4$. Equality holds exactly for the arcsine measure
\[
 \dd\xi(x)=\frac{\dd x}{\pi\sqrt{x(1-x)}}.
\]
\end{lemma}
\begin{proof}
Lift $\xi$ symmetrically to a circle probability measure $\Lambda$ using
$x=(1+\cos\theta)/2$. The identity
\[
 |x-y|=\frac14|e^{i\theta}-e^{i\varphi}|
                  |e^{i\theta}-e^{-i\varphi}|
\]
gives $I(\xi)=2I_{\T}(\Lambda)-\log4$, where
$I_{\T}(\Lambda)=\iint\log|\zeta-\eta|\,\dd\Lambda(\zeta)\,\dd\Lambda(\eta)$.
This identity also holds for energy $-\infty$, by integration of kernels bounded above.
For $0<r<1$, the absolutely convergent Fourier series gives
\[
 I_{\T,r}:=\iint\log|1-r\zeta\bar\eta|\,\dd\Lambda(\zeta)\,\dd\Lambda(\eta)
 =-\sum_{k\geq1}\frac{r^k}{k}|\widehat\Lambda(k)|^2\leq0.
\]
Since $|1-r\zeta\bar\eta|^2=(1-r)^2+r|\zeta-\eta|^2$,
$I_{\T}(\Lambda)\leq I_{\T,r}-\tfrac12\log r$. Letting $r\uparrow1$ proves $I_{\T}(\Lambda)\leq0$.
If $I_{\T}(\Lambda)=0$, the same inequality gives $I_{\T,r}\to0$. Each nonpositive summand then
forces $\widehat\Lambda(k)=0$ for $k\geq1$, so $\Lambda$ is uniform.
Conversely, normalized Lebesgue measure on the circle has energy zero, as follows from
$\int_0^{2\pi}\log|1-e^{i\theta}|\,\dd\theta=0$.
Its image on $[0,1]$ is the stated arcsine measure.
\end{proof}

\begin{theorem}\label{thm:arcsine}
As $n\to\infty$,
\begin{equation}\label{eq:arcsine-limit}
 \xi_n:=\sum_{k=0}^n v_{n,k}\delta_{k/n}
 \Longrightarrow\frac{\dd x}{\pi\sqrt{x(1-x)}},
\end{equation}
and
\begin{equation}\label{eq:critical-asymptotic}
 t_n=\log n-\log4+o(1),\qquad nr(n)\longrightarrow8\pi.
\end{equation}
\end{theorem}
\begin{proof}
The comparison $H_\infty(t)\leq e^{-t/2}$ and
\cite[Propositions 2.7 and 2.12]{Chenevier2020} give
\begin{equation}\label{eq:critical-lower}
 t_n\geq\psi(1/2)+\sum_{j=1}^n\frac1j
 =\log n-\log4+o(1).
\end{equation}
Put $a_d=L+\psi((d+1)/2)$. The inequality $\psi(u)<\log u$ for $u>0$ implies
$a_d<\log(d+1)$. Indeed, a gamma-distributed variable of shape $u$ and rate one has mean $u$ and mean
logarithm $\psi(u)$, so strict concavity of the logarithm gives the inequality.
By positivity of all $v_{n,k}$,
\begin{equation}\label{eq:discrete-energy}
 t_n-\log n\leq
 \iint\log(|x-y|+1/n)\,\dd\xi_n(x)\,\dd\xi_n(y).
\end{equation}
By compactness, every subsequence of $(\xi_n)$ has a further weakly convergent
subsequence. Fix one with limit $\xi$. For fixed $\varepsilon>0$,
$1/n\leq\varepsilon$ eventually, and the kernel
$\log(|x-y|+\varepsilon)$ is continuous on $[0,1]^2$. Thus along this subsequence,
\[
 \limsup_{n\to\infty}(t_n-\log n)
 \leq\iint\log(|x-y|+\varepsilon)\,\dd\xi(x)\,\dd\xi(y).
\]
Letting $\varepsilon\downarrow0$ gives the upper bound $I(\xi)$, by monotone convergence
after subtracting from a fixed upper bound. Together with \eqref{eq:critical-lower} and
\cref{lem:energy}, this yields
\[
 -\log4\leq\liminf(t_n-\log n)
 \leq\limsup(t_n-\log n)\leq I(\xi)\leq-\log4.
\]
Equality forces $\xi$ to be the arcsine measure. Every subsequential limit is therefore
the same, proving \eqref{eq:arcsine-limit} and the first assertion in
\eqref{eq:critical-asymptotic}. Exponentiation gives the second.
\end{proof}

\Cref{prop:monotone,thm:arcsine} prove \cref{thm:intro-limits}.
The convergence in \eqref{eq:arcsine-limit} is weak convergence of measures; it does not
assert a uniform pointwise asymptotic for individual coordinates.

The boundary behaviour depends on the rationality of $v_n$.
Chenevier proved rationality in odd degree and noted its failure in positive even degree
\cite[Corollary 2.16 and the proof of Proposition 6.5]{Chenevier2020}.
The moment formula gives a short proof of the latter assertion.

\begin{proposition}\label{prop:rational}
All coordinates of $v_n$ are rational if and only if $n=0$ or $n$ is odd.
\end{proposition}
\begin{proof}
The positive assertions follow from $v_0=(1)$ and the cited corollary.
Suppose $n=2m\geq2$ and $v_n\in\Q^{n+1}$. By the Laurent--Chebyshev correspondence,
$P_{2m}(x)=\sum_{j=0}^m a_jx^{2j}$ has rational coefficients.
Its constant coefficient is nonzero: the corresponding monic polynomial has value
$(-1)^m\beta_1\beta_3\cdots\beta_{2m-1}\ne0$ at zero.
Orthogonality to the constant polynomial gives
\[
 a_0L+\sum_{j=1}^m a_js_j=0.
\]
Since each $s_j$ with $j\geq1$ is rational, this contradicts the irrationality of $\log2$.
\end{proof}

We use Chenevier's class of test functions: even real functions satisfying (TFa)--(TFc)
in \cite[Section 4.2]{Chenevier2020}. Explicitly, $F(x)e^{\varepsilon|x|}$ is integrable
and of bounded variation for some $\varepsilon>1/2$;
$(F(x)-F(0))/x$ is of bounded variation on $\R\setminus\{0\}$; and $F$ equals the
arithmetic mean of its one-sided limits at each point. An \emph{admissible} function
below is a nonzero nonnegative such $F$ satisfying POS, equivalently
\[
 \widehat G(u)\geq0\quad(u\in\R),\qquad G(x)=F(x)\cosh(x/2),
 \quad \widehat G(u)=\int_\R G(x)e^{-2\pi i xu}\,\dd x.
\]
We normalize it by $F(0)=1$, as allowed by \cite[Lemma 4.5]{Chenevier2020}.
For $H_F(t)=e^{-t/2}F(t)$ on $[0,\infty)$, put
\[
 \begin{aligned}
 \psi_{H_F}(s)&=\int_0^\infty
 \left(H_F(0)\frac{e^{-t}}t-\frac{H_F(t)e^{-st}}{1-e^{-t}}\right)\dd t,
 \qquad s\geq0,\\
 q_n^{H_F}(x)&=\sum_{i,j=0}^n x_ix_j\psi_{H_F}(|i-j|).
 \end{aligned}
\]
For every real $r>0$, define
\begin{equation}\label{eq:boundary-forms}
 Q_{F,r}(x)=\log(2\pi/r)\,\phi_n(x)^2-q_n^{H_F}(x),\qquad
 Q_{\infty,r}(x)=\log(2\pi/r)\,\phi_n(x)^2-q_n(x).
\end{equation}
The first is Chenevier's complex-place quadratic form with the $\log r$ dimension
term subtracted \cite[Lemma 3.4]{Chenevier2020}. Here $r$ is an arbitrary positive
real parameter, not necessarily a root discriminant.

\begin{theorem}[A common integral obstruction]\label{thm:boundary}
At $r=r(n)$ there is a single nonzero $x\in\mathbb Z_{\geq0}^{n+1}$ such that
\begin{equation}\label{eq:common-negative}
 Q_{F,r(n)}(x)<0\quad\hbox{for every normalized admissible }F
\end{equation}
if and only if $n=0$ or $n$ is odd. In those cases $x$ can be chosen strictly positive
and symmetric, with $x_k=x_{n-k}$.
For positive even $n$, each individual $F$ has a strictly positive symmetric integral
negative vector, but no common nonzero integral vector exists.
\end{theorem}
\begin{proof}
By \cite[Lemma 4.5]{Chenevier2020}, $F(t)\leq F_\infty(t)$. Hence
$\Delta=H_\infty-H_F\geq0$ and $\Delta(0)=0$.
Moreover, $\Delta$ is nonzero in $L^1$: equality almost everywhere with $F_\infty$
would contradict the exponential integrability required by (TFa).
Proposition 2.2 of \cite{Chenevier2020}, applied to $\Delta$, shows that
$q_n^\Delta$ is negative definite on the whole space. By linearity in $H$,
\begin{equation}\label{eq:strict-test-comparison}
 Q_{\infty,r}(x)-Q_{F,r}(x)=-q_n^\Delta(x)>0\qquad(x\ne0).
\end{equation}
At $r=r(n)$, \eqref{eq:variational} shows that $Q_{\infty,r(n)}$ is positive semidefinite
with kernel $\R v_n$. Thus $Q_{F,r(n)}(v_n)<0$ for every admissible $F$.
When $v_n$ is rational, a positive integral multiple gives the required common vector.

Conversely, if a nonzero integral vector lies on $\R v_n$, normalizing it by its nonzero
coordinate sum makes $v_n$ rational. If $v_n$ is not rational, every nonzero integral
vector $x$ therefore satisfies $Q_{\infty,r(n)}(x)>0$.
The normalized Odlyzko family $F_\lambda$ of \cite[Section 4.3]{Chenevier2020} satisfies
$Q_{F_\lambda,r(n)}(x)\to Q_{\infty,r(n)}(x)$ for each fixed $x$
\cite[Lemma 5.1]{Chenevier2020}. Thus $Q_{F_\lambda,r(n)}(x)>0$ for large $\lambda$, excluding
\eqref{eq:common-negative}.
For an individual $F$, strict negativity at $v_n$, continuity, and $v_n>0$ give a strictly
positive symmetric rational approximation still having negative value. Clearing
its denominators gives an integral vector. Apply \cref{prop:rational} to finish.
\end{proof}

For any $r>r(n)$, \eqref{eq:variational} gives
$Q_{\infty,r}(v_n)=\log(r(n)/r)<0$. Rational approximation and
\eqref{eq:strict-test-comparison} then provide a common strictly positive symmetric
integral negative vector in every degree. Chenevier's corresponding effective-cone
statement assumes $r\geq1$ \cite[Proposition 6.5]{Chenevier2020}; strict positivity permits
this quadratic-form statement for arbitrary $r>r(n)>0$.

\section{Linear response and recurrence asymptotics}\label{sec:response}

The circle density in \eqref{eq:circle-density}, up to a positive constant, is
\[
  W(\theta)=(\pi-|\theta|)\tan\frac{|\theta|}{2}.
\]
It factors as
\begin{equation}\label{eq:FH-factorization}
  W(\theta)=|1-e^{i\theta}|e^{V(\theta)},
\end{equation}
where
\begin{equation}\label{eq:V-def}
  V(\theta)=\log\frac{\pi-|\theta|}{2\cos(\theta/2)},
  \qquad -\pi<\theta<\pi.
\end{equation}
The first factor is a Fisher--Hartwig singularity with root exponent $1/2$, and
\begin{equation}\label{eq:V-cusp}
  V(\theta)=V(0)-\frac{|\theta|}{\pi}+O(\theta^2)
  \qquad(\theta\to0).
\end{equation}
Thus the background is continuous but not differentiable at the singular point.

Write $V(\theta)=V_0+2\sum_{k\geq1}V_k\cos(k\theta)$, where for $k\geq0$ we set
\begin{equation}\label{eq:Vk-def}
  V_k=\frac1\pi\int_0^\pi V(\theta)\cos(k\theta)\,\dd\theta.
\end{equation}

\begin{theorem}\label{thm:fourier-V}
For every $k\geq1$,
\begin{equation}\label{eq:Vk-Si}
  V_k=\frac{(-1)^{k+1}}{\pi k}
  \left(\Si(k\pi)-\frac\pi2\right)>0.
\end{equation}
For every fixed $M\geq1$,
\begin{equation}\label{eq:Vk-asymptotic-full}
  V_k=
  \sum_{m=0}^{M-1}
  \frac{(-1)^m(2m)!}{\pi^{2m+2}k^{2m+2}}
  +O_M(k^{-2M-2}).
\end{equation}
In particular,
\begin{equation}\label{eq:Vk-first}
  V_k=\frac1{\pi^2k^2}-\frac2{\pi^4k^4}
  +\frac{24}{\pi^6k^6}+O(k^{-8}).
\end{equation}
\end{theorem}

\begin{proof}
The additive constant in $V$ does not contribute. Substituting $u=\pi-\theta$ and
integrating by parts gives
\[
  \int_0^\pi\log(\pi-\theta)\cos(k\theta)\,\dd\theta
  =\frac{(-1)^{k+1}}{k}\Si(k\pi).
\]
The Fourier series
\[
  \log\left(2\cos\frac\theta2\right)
  =\sum_{m\geq1}\frac{(-1)^{m+1}}m\cos(m\theta)
\]
gives the remaining term in \eqref{eq:Vk-Si}. The asymptotic series for $\Si$ at integer multiples of $\pi$
\cite[Section 6.12(ii)]{DLMF} gives \eqref{eq:Vk-asymptotic-full}. Finally, the tail representation of the sine integral gives
\begin{equation}\label{eq:kVk-integral}
  kV_k=\frac1\pi\int_0^\infty\frac{\sin(kv)}{\pi+v}\,\dd v
  =\frac1\pi\int_0^\infty e^{-\pi t}\frac{k}{k^2+t^2}\,\dd t>0,
\end{equation}
where the first integral is an improper Dirichlet integral. To justify the second equality,
insert an Abel factor $e^{-\varepsilon v}$ and use
\[
  \frac1{\pi+v}=\int_0^\infty e^{-(\pi+v)t}\,\dd t.
\]
Then apply Fubini for $\varepsilon>0$. The resulting integrand is
$e^{-\pi t}k/(k^2+(t+\varepsilon)^2)$; dominated convergence as
$\varepsilon\downarrow0$ gives \eqref{eq:kVk-integral} and proves positivity.
\end{proof}

We first compute the response to a continuous even real-valued perturbation $B$ on
$[-1,1]$. For $t\in\R$, set
\begin{equation}\label{eq:eta-general}
 \dd\eta_t^B(x)=\frac{e^{tB(x)}}{\int_{-1}^1e^{tB(u)}\,\dd u}\,\dd x.
\end{equation}
Use the circle lift of \cref{sec:positivity}, denote its Verblunsky coefficients by
$\alpha_j^B(t)$, and set
\begin{equation}\label{eq:response-functional}
 \mathscr R_n[B]=\left.\frac{\mathrm d}{\mathrm dt}\alpha_{n-1}^B(t)\right|_{t=0},
 \qquad n\geq1.
\end{equation}
For $B=B_*$ we write $\alpha_j(t)=\alpha_j^{B_*}(t)$, as in the introduction;
the unadorned $\alpha_j$ always denotes $\alpha_j(1)$.
At $t=0$ the real-line polynomials are the Legendre polynomials $\mathsf P_n(1)=1$.
Since $c_n^{(0)}=n/(2n+1)$, \eqref{eq:alpha-c} gives
\begin{equation}\label{eq:alpha-legendre}
 \alpha_{n-1}^B(0)=-\frac1{2n+1}.
\end{equation}
We use the Legendre operator and the arcsine average
\begin{equation}\label{eq:Legendre-operator}
 \mathcal Lq=-\frac{\mathrm d}{\mathrm dx}\bigl((1-x^2)q'(x)\bigr),\qquad
 \mathcal A(q)=\frac1\pi\int_{-1}^1\frac{q(x)}{\sqrt{1-x^2}}\,\dd x.
\end{equation}

\begin{theorem}[Legendre first variation]\label{thm:linear-response}
For every continuous even real-valued $B$ and every $n\geq1$,
\begin{equation}\label{eq:general-response}
 \mathscr R_n[B]=\frac1{2(2n+1)}\int_{-1}^1
 B(x)\mathcal L\bigl(\mathsf P_n(x)^2\bigr)\,\dd x.
\end{equation}
If, in addition, $B\in C^2((-1,1))$, $\mathcal LB\in L^1(-1,1)$, and
$(1-x^2)B'(x)\to0$ at both endpoints, then
\begin{equation}\label{eq:general-response-LB}
 \mathscr R_n[B]=\frac1{2(2n+1)}\int_{-1}^1
 \mathsf P_n(x)^2\mathcal LB(x)\,\dd x.
\end{equation}
\end{theorem}
\begin{proof}
Let $P_m^B(x;t)$ be the orthogonal polynomials for \eqref{eq:eta-general}, normalized by
$P_m^B(1;t)=1$, and let $\ell_m^B(t)$ be their leading coefficients.
The finite Gram matrices have smooth entries in $t$ and are positive definite, so these
polynomials and their coefficients depend smoothly on $t$. Write $c_m^B(t)$ for the
endpoint recurrence parameters and put
\[
 U_m=\left.\frac{\mathrm d}{\mathrm dt}\log\ell_m^B(t)\right|_{t=0}.
\]
As $1-c_m^B(t)=\ell_m^B(t)/\ell_{m+1}^B(t)$, \eqref{eq:alpha-c} yields
\begin{equation}\label{eq:response-leading}
 \mathscr R_n[B]=\frac{2(n+1)}{2n+1}(U_{n+1}-U_n).
\end{equation}
Write a dot for differentiation at $t=0$. The expansion
\[
 \dot P_m^B=U_m\mathsf P_m+\sum_{r=0}^{m-1}b_{m,r}\mathsf P_r
\]
satisfies $U_m+\sum_{r<m}b_{m,r}=0$, by endpoint normalization.
Differentiate orthogonality against $\mathsf P_r$, using
$\int_{-1}^1\mathsf P_r^2\,\dd x/2=(2r+1)^{-1}$. The constant arising from normalization
of the measure contributes zero, and
\[
 b_{m,r}=-\frac{2r+1}{2}\int_{-1}^1 B(x)\mathsf P_m(x)\mathsf P_r(x)\,\dd x.
\]
Consequently,
\begin{equation}\label{eq:leading-kernel}
 U_m=\frac12\int_{-1}^1 B(x)\mathsf P_m(x)K_{m-1}(1,x)\,\dd x,
 \qquad K_m(1,x)=\sum_{r=0}^m(2r+1)\mathsf P_r(x).
\end{equation}
Combining \eqref{eq:response-leading} and \eqref{eq:leading-kernel} expresses the response
as the integral of $B$ against
\[
 \frac{n+1}{2n+1}
   \bigl(\mathsf P_{n+1}K_n(1,\cdot)-\mathsf P_nK_{n-1}(1,\cdot)\bigr).
\]
Only the even part contributes.
Christoffel--Darboux gives
\[
 K_m(1,x)=\frac{(m+1)(\mathsf P_{m+1}(x)-\mathsf P_m(x))}{x-1}.
\]
Set $Y_n=(1-x^2)\mathsf P_n'$,
$E=(n+1)\mathsf P_{n+1}^2-n\mathsf P_n^2$, and
$O=n\mathsf P_n\mathsf P_{n-1}-(n+1)\mathsf P_{n+1}\mathsf P_n$.
Then $E$ is even, $O$ is odd, and
\[
 \operatorname{Even}\!\left(\frac{E+O}{x-1}\right)=\frac{E+xO}{x^2-1},
\]
where $\operatorname{Even}q(x)=(q(x)+q(-x))/2$. The identities
\[
 \mathsf P_{n-1}=x\mathsf P_n+\frac{Y_n}{n},\qquad
 \mathsf P_{n+1}=x\mathsf P_n-\frac{Y_n}{n+1}
\]
give
\[
 E+xO=-n(1-x^2)\mathsf P_n^2+\frac{Y_n^2}{n+1}.
\]
Thus the required even part is
\[
 \frac{n(n+1)\mathsf P_n^2-(1-x^2)(\mathsf P_n')^2}{2n+1}
 =\frac1{2(2n+1)}\mathcal L(\mathsf P_n^2),
\]
which proves \eqref{eq:general-response}.
Under the additional assumptions, Green's identity on compact subintervals can be
passed to the endpoints. The terms
$(1-x^2)B'(x)\mathsf P_n(x)^2$ and
$(1-x^2)B(x)(\mathsf P_n(x)^2)'$ both vanish there, and
$\mathcal LB$ is integrable. This proves \eqref{eq:general-response-LB}.
\end{proof}

The endpoint contribution can be evaluated through two Legendre averages. For $j,n\geq0$, put
\begin{equation}\label{eq:omega-kappa}
 \omega_j=\frac{\binom{2j}{j}^2}{16^j},\qquad
 \kappa_n=\sum_{j=0}^n \omega_j\omega_{n-j}.
\end{equation}

\begin{lemma}\label{lem:Legendre-averages}
For every $n\geq0$, $\mathcal A(\mathsf P_n^2)=\kappa_n$, and, as $n\to\infty$,
\begin{equation}\label{eq:kappa-asymptotic}
 \kappa_n=\frac2{\pi^2n}\bigl(\log n+\gamma+\log16\bigr)
       +O\left(\frac{\log n}{n^2}\right).
\end{equation}
For every $q\in C([-1,1])$,
\begin{equation}\label{eq:Legendre-weak}
 n\int_{-1}^1\mathsf P_n(x)^2q(x)\,\dd x\longrightarrow\mathcal A(q).
\end{equation}
\end{lemma}
\begin{proof}
The Legendre generating function gives
\[
 e^{in\theta}\mathsf P_n(\cos\theta)
 =\sum_{j=0}^n\frac{\binom{2j}{j}\binom{2n-2j}{n-j}}{4^n}e^{2ij\theta}.
\]
Parseval's identity on $[0,\pi]$ proves $\mathcal A(\mathsf P_n^2)=\kappa_n$.
To evaluate the convolution, put $u_0=0$, $u_j=1/(\pi j)$ for $j\geq1$, and
$\varepsilon_j=\omega_j-u_j$. Stirling's formula gives $\varepsilon_j=O((j+1)^{-2})$.
For the complete elliptic integral with modulus $k$,
\[
 K(k)=\int_0^{\pi/2}\frac{\dd\theta}{\sqrt{1-k^2\sin^2\theta}},
\]
its Taylor series and the expansion at $k=1$ give
\[
 \sum_{j\geq0}\omega_jz^j=\frac2\pi K(\sqrt z)
 =\frac1\pi\log\frac{16}{1-z}+o(1)\qquad(z\uparrow1);
\]
see \cite[Section 19.12]{DLMF}. Subtraction of
$\sum u_jz^j=-\pi^{-1}\log(1-z)$ and absolute summability imply
$\sum_{j\geq0}\varepsilon_j=(\log16)/\pi$.
For the convolution $(u*v)_n=\sum_{j=0}^n u_jv_{n-j}$, one has
\[
 (u*u)_n=\frac{2H_{n-1}}{\pi^2n},\quad
 (u*\varepsilon)_n=\frac{\log16}{\pi^2n}+O(\log n/n^2),\quad
 (\varepsilon*\varepsilon)_n=O(n^{-2}).
\]
For the middle estimate, write
$(u*\varepsilon)_n=\pi^{-1}\sum_{j=0}^{n-1}\varepsilon_j/(n-j)$, split at $n/2$, and use
$(n-j)^{-1}-n^{-1}=j/(n(n-j))$ on the first part and $\varepsilon_j=O(n^{-2})$ on the
second. The last estimate follows by the same split and summability of $\varepsilon_j$.
Since $\kappa_n=(u*u)_n+2(u*\varepsilon)_n+(\varepsilon*\varepsilon)_n$,
$H_{n-1}=\log n+\gamma+O(n^{-1})$ proves \eqref{eq:kappa-asymptotic}.

For \eqref{eq:Legendre-weak}, use the orthonormal Legendre polynomials
$\mathsf p_n=\sqrt{(2n+1)/2}\,\mathsf P_n$. Their recurrence is
\[
 x\mathsf p_n=a_{n+1}^{(0)}\mathsf p_{n+1}+a_n^{(0)}\mathsf p_{n-1},\qquad
 a_n^{(0)}=\frac{n}{\sqrt{4n^2-1}}\longrightarrow\frac12\quad(n\geq1),
\]
with $a_0^{(0)}=0$. For a fixed integer $m\geq0$, expanding $x^m\mathsf p_n$ by the recurrence shows that the
coefficient of $\mathsf p_n$ tends to zero for odd $m$ and to
$2^{-2r}\binom{2r}{r}$ for $m=2r$. There are $\binom{2r}{r}$ paths with equal numbers
of upward and downward steps, and the weight of each tends to $2^{-2r}$. These are exactly the moments of the arcsine measure on $[-1,1]$.
Thus \eqref{eq:Legendre-weak} holds for polynomials, since $2n/(2n+1)\to1$.
Uniform polynomial approximation extends the limit to continuous $q$, using
$n\int\mathsf P_n^2\,\dd x=2n/(2n+1)\leq1$.
\end{proof}

\begin{theorem}[Linear response to an endpoint cusp]\label{thm:endpoint-response}
Let $B\in C([-1,1])\cap C^2((-1,1))$ be even and real-valued, with
$(1-x^2)B'(x)\to0$ at both endpoints. Assume that
\begin{equation}\label{eq:general-cusp}
 \mathcal LB(x)=\frac{a}{\sqrt{1-x^2}}+J_B(x),\qquad J_B\in C([-1,1]),
 \quad a\in\R.
\end{equation}
Then
\begin{equation}\label{eq:general-cusp-response}
 \mathscr R_n[B]
 =\frac{a}{2\pi}\frac{\log n}{n^2}
 +\frac1{n^2}\left(\frac{a(\gamma+\log16)}{2\pi}
                  +\frac14\mathcal A(J_B)\right)+o(n^{-2}).
\end{equation}
\end{theorem}
\begin{proof}
The hypotheses imply $\mathcal LB\in L^1(-1,1)$, so \eqref{eq:general-response-LB} gives
\begin{equation}\label{eq:general-exact-convolution}
 \mathscr R_n[B]=\frac{a\pi \kappa_n}{2(2n+1)}
 +\frac1{2(2n+1)}\int_{-1}^1\mathsf P_n(x)^2J_B(x)\,\dd x.
\end{equation}
Apply both limits in \cref{lem:Legendre-averages}. The first summand equals
$a(\log n+\gamma+\log16)/(2\pi n^2)+o(n^{-2})$, while the second equals
$\mathcal A(J_B)/(4n^2)+o(n^{-2})$.
\end{proof}

We now apply the theorem to $B_*$ from \eqref{eq:intro-deformation}.
For $0<|x|<1$, direct differentiation gives
\begin{equation}\label{eq:LBstar}
 \mathcal LB_*(x)=\frac2{\pi\sqrt{1-x^2}}+J_*(x),
\end{equation}
where
\begin{equation}\label{eq:J-explicit}
 J_*(x)=\frac{x}{\sqrt{1-x^2}\arcsin x}
      +\frac1{(\arcsin x)^2}-\frac1{x^2}-1
      -\frac2{\pi\sqrt{1-x^2}}.
\end{equation}
The function $J_*$ extends continuously to $[-1,1]$, with
\begin{equation}\label{eq:J-endpoints}
 J_*(0)=-\frac13-\frac2\pi,\qquad J_*(1)=J_*(-1)=\frac8{\pi^2}-2.
\end{equation}
Indeed, at zero $\mathcal LB_*=-1/3+4x^2/15+O(x^4)$.
At $x=\cos\theta$ with $\theta\downarrow0$, the identity
$\arcsin x=\pi/2-\theta$ gives
\[
 \mathcal LB_*(\cos\theta)=\frac2{\pi\sin\theta}+\frac8{\pi^2}-2+O(\theta).
\]
Evenness treats the other endpoint. Moreover,
$B_*'=1/(\sqrt{1-x^2}\arcsin x)-1/x$ and
$(1-x^2)B_*'=O(\sqrt{1-|x|})$ at the endpoints, so all hypotheses hold.

\begin{corollary}[Response to the Chenevier weight]\label{cor:full-background}
For the deformation \eqref{eq:intro-deformation},
\begin{equation}\label{eq:full-background-asymptotic}
 \dot\alpha_{n-1}(0)
 =\frac{\log n}{\pi^2n^2}+\frac{C_{\mathrm{lin}}}{n^2}+o(n^{-2}),
\end{equation}
where
\begin{equation}\label{eq:Clinear}
 C_{\mathrm{lin}}=\frac{\gamma+\log16}{\pi^2}
  +\frac1{4\pi}\int_{-1}^1\frac{J_*(x)}{\sqrt{1-x^2}}\,\dd x.
\end{equation}
There is also the exact identity
\begin{equation}\label{eq:full-response-exact}
 \dot\alpha_{n-1}(0)=\frac{\kappa_n}{2n+1}
  +\frac1{2(2n+1)}\int_{-1}^1\mathsf P_n(x)^2J_*(x)\,\dd x.
\end{equation}
\end{corollary}
\begin{proof}
Use $a=2/\pi$ and $J_B=J_*$ in \cref{thm:endpoint-response} and
\eqref{eq:general-exact-convolution}.
\end{proof}

This proves \cref{thm:intro-response}. Numerical quadrature of the convergent integral in
\eqref{eq:Clinear} gives $C_{\mathrm{lin}}\approx0.08225251457$.
The formula also sums the Fourier response rigorously. The identity
$B_*(\cos\theta)=V(2\theta)$, with $V$ read periodically, gives the uniformly convergent
series
\[
 B_*(x)=V_0+2\sum_{k\geq1}V_kT_{2k}(x).
\]
For fixed $n$, \eqref{eq:general-response} is a bounded linear functional of $B$ and
annihilates constants. Hence, for $R_{n,k}=\mathscr R_n[2T_{2k}]$ with $n,k\geq1$,
\begin{equation}\label{eq:summed-mode-identity}
 \dot\alpha_{n-1}(0)=\sum_{k\geq1}V_kR_{n,k}.
\end{equation}
For each fixed $n$, the boundedness of $\mathscr R_n$ and $\|T_{2k}\|_\infty=1$
give $|R_{n,k}|\leq 2\|\mathscr R_n\|$. Since $\sum_{k\geq1}|V_k|<\infty$,
the series is absolutely convergent. Thus \cref{cor:full-background}
evaluates the full Fourier response without interchanging an asymptotic expansion and an
infinite sum. Individual-mode expansions are given in Appendix~\ref{app:response}.

To pass from the derivative at $t=0$ to the original weight at $t=1$, one must estimate
the nonlinear remainder
\begin{equation}\label{eq:nonlinear-remainder}
 \mathcal N_n=\alpha_{n-1}(1)-\alpha_{n-1}(0)-\dot\alpha_{n-1}(0).
\end{equation}
An estimate $\mathcal N_n=O(n^{-2})$ would suffice to prove the leading logarithmic
correction at $t=1$. A limit for $n^2\mathcal N_n$ would additionally identify the
next constant. The following more precise assertion remains conjectural.

\begin{conjecture}\label{conj:alpha-asymptotic}
There exists $C_\alpha\in\R$ such that
\begin{equation}\label{eq:alpha-conjecture}
 \alpha_{n-1}
 =-\frac1{2n+1}+\frac{\log n}{\pi^2n^2}+\frac{C_\alpha}{n^2}
  +O\left(\frac{\log^2n}{n^3}\right).
\end{equation}
\end{conjecture}

In particular, $C_\alpha=C_{\mathrm{lin}}$ would require $\mathcal N_n=o(n^{-2})$.
The smooth-background Fisher--Hartwig results of \cite{DeiftItsKrasovsky2011} do not
apply directly without verifying their regularity assumptions at the cusp
\eqref{eq:V-cusp}.

The exact monic recurrence relation is
\begin{equation}\label{eq:beta-alpha}
 \beta_n=\frac14(1-\alpha_{n-2})(1+\alpha_{n-1}),\qquad n\geq2.
\end{equation}
An error estimate for $\alpha_{n-1}$ alone does not control its first difference.
The following implication states the additional regularity needed for $\beta_n$.

\begin{proposition}\label{prop:conditional-beta}
Suppose that, as $n\to\infty$,
\[
  \alpha_{n-1}=-\frac1{2n+1}
  +\frac{\log n}{\pi^2n^2}+e_n,
  \qquad
  e_n=O(n^{-2}),
  \qquad
  e_n-e_{n-1}=O(n^{-3}).
\]
Then
\begin{equation}\label{eq:beta-log-consequence}
  \beta_n
  =\frac{n^2}{4n^2-1}
  -\frac{\log n}{4\pi^2n^3}+O(n^{-3}).
\end{equation}
\end{proposition}

\begin{proof}
Write
\[
  \alpha_{n-1}=-\frac1{2n+1}+\delta_n,
  \qquad
  \delta_n=h_n+e_n,
  \qquad
  h_n=\frac{\log n}{\pi^2n^2},
\]
and substitute into \eqref{eq:beta-alpha}. The unperturbed terms give
$n^2/(4n^2-1)$. The linear perturbation is
\[
  \frac14\left[
  \left(1+\frac1{2n-1}\right)\delta_n
  -\left(1-\frac1{2n+1}\right)\delta_{n-1}
  \right].
\]
Taylor expansion of the explicit sequence $h_n$ gives
\[
  \left(1+\frac1{2n-1}\right)h_n
  -\left(1-\frac1{2n+1}\right)h_{n-1}
  =-\frac{\log n}{\pi^2n^3}+O(n^{-3}).
\]
The assumptions on $e_n$ give
\[
  \begin{aligned}
  &\left(1+\frac1{2n-1}\right)e_n
  -\left(1-\frac1{2n+1}\right)e_{n-1}\\
  &\qquad=(e_n-e_{n-1})
  +O\!\left(\frac{|e_n|+|e_{n-1}|}{n}\right)
  =O(n^{-3}).
  \end{aligned}
\]
Finally, the quadratic perturbation is
$-\delta_{n-1}\delta_n/4=O(\log^2n/n^4)$. Combining the three estimates proves
\eqref{eq:beta-log-consequence}.
\end{proof}

\appendix
\section{Fourier-mode expansions and endpoint sums}\label{app:response}

At the baseline, the circle measure is
\[
 \dd\sigma_0(e^{i\varphi})=\frac18|1-e^{i\varphi}|\,\dd\varphi,
 \qquad 0\leq\varphi<2\pi.
\]
Multiplication by $e^{t(z^k+z^{-k})}$, followed by normalization, corresponds on the
interval to the perturbation $B=2T_{2k}$. Therefore \cref{thm:linear-response} gives
\begin{equation}\label{eq:Rnk-Legendre}
 R_{n,k}=\frac1{2n+1}\int_{-1}^1
   \mathsf P_n(x)^2\mathcal LT_{2k}(x)\,\dd x.
\end{equation}
Equivalently,
\begin{equation}\label{eq:Rnk-Gn}
 R_{n,k}=\int_{-1}^1 T_{2k}(x)G_n(x)\,\dd x,
\end{equation}
where
\begin{equation}\label{eq:Gn}
 G_n=\frac1{2n+1}\mathcal L(\mathsf P_n^2)
 =\frac2{2n+1}\left(n(n+1)\mathsf P_n^2-(1-x^2)(\mathsf P_n')^2\right).
\end{equation}
Expand the polynomial $\mathcal LT_{2k}$ in the Legendre basis:
\begin{equation}\label{eq:L-T-expansion}
  \mathcal LT_{2k}=\sum_{j=1}^kd_{k,j}\mathsf P_{2j}
\end{equation}
since $\mathcal LT_{2k}$ is even of degree $2k$ and, by the symmetry of $\mathcal L$,
is orthogonal to the constant polynomial. Set
\begin{equation}\label{eq:Inj-omega}
  I_{n,j}=\int_{-1}^1\mathsf P_n(x)^2\mathsf P_{2j}(x)\,\dd x.
\end{equation}
We retain the weights $\omega_j$ from \eqref{eq:omega-kappa}.
Orthogonality gives $I_{n,j}=0$ when $j>n$. For $0\leq j\leq n$, the Gaunt integral is
\begin{equation}\label{eq:Gaunt-explicit}
  I_{n,j}=2\frac{(n+j)!^2(2j)!^2(2n-2j)!}
  {j!^4(n-j)!^2(2n+2j+1)!}.
\end{equation}
In particular,
\begin{equation}\label{eq:Gaunt-product}
  \frac{I_{n,j}}{\omega_j}
  =\frac{2}{2n+1}\prod_{r=0}^{j-1}
  \frac{4(n-r)(n+r+1)}{(2n-2r-1)(2n+2r+3)}.
\end{equation}

\begin{lemma}\label{lem:Gaunt-expansion}
For every fixed $j\geq0$, put $\lambda_j=2j(2j+1)$. Then, as $n\to\infty$,
\begin{equation}\label{eq:Inj-expansion}
  \begin{aligned}
  \frac{I_{n,j}}{\omega_j}=\frac1n\Bigg[&1-\frac1{2n}
  +\frac{\lambda_j+2}{8n^2}
  -\frac{3\lambda_j+2}{16n^3}\\
  &+\frac{3\lambda_j^2+22\lambda_j+8}{128n^4}
  +O_j(n^{-5})\Bigg].
  \end{aligned}
\end{equation}
\end{lemma}

\begin{proof}
The Legendre product formula and the explicit $3j$ symbol with zero lower indices
\cite[equations (34.3.19) and (34.3.5)]{DLMF} give \eqref{eq:Gaunt-explicit}.
Dividing by $\omega_j$ and cancelling factorials gives \eqref{eq:Gaunt-product}. For fixed $j$, take logarithms in that
finite product and expand each factor in powers of $1/n$. With
$\lambda_j=2j(2j+1)$, summation over $r=0,\ldots,j-1$ gives
\[
  \log\!\left(\frac{nI_{n,j}}{\omega_j}\right)
  =-\frac1{2n}+\frac{\lambda_j+1}{8n^2}
  -\frac{3\lambda_j+1}{24n^3}
  +\frac{\lambda_j^2+5\lambda_j+1}{64n^4}
  +O_j(n^{-5}).
\]
Exponentiating this displayed formula gives \eqref{eq:Inj-expansion}.
\end{proof}

The functional $\mathcal A$ in \eqref{eq:Legendre-operator} satisfies
$\mathcal A(\mathsf P_{2j})=\omega_j$. Elementary trigonometric calculation gives
\begin{equation}\label{eq:L-moments}
  \mathcal A(\mathcal LT_{2k})=2k,
  \qquad
  \mathcal A(\mathcal L^2T_{2k})=12k^3,
\end{equation}
\begin{equation}\label{eq:L3-moment}
  \mathcal A(\mathcal L^3T_{2k})=12k^3(5k^2+1).
\end{equation}
Indeed,
\[
  \mathcal LT_m(\cos\theta)
  =m^2\cos(m\theta)+m\cot\theta\sin(m\theta),
\]
and, for $r\geq1$,
\[
  \cot\theta\sin(2r\theta)
  =1+\cos(2r\theta)+2\sum_{\ell=1}^{r-1}\cos(2\ell\theta).
\]
Consequently,
\[
  \mathcal LT_{2r}
  =2r+4r\sum_{\ell=1}^{r-1}T_{2\ell}+2r(2r+1)T_{2r}.
\]
Using $\mathcal A(T_j)=0$ for $j\geq1$ and applying this triangular identity once and
twice more, together with the standard sums of first and third powers, gives for even
$m\geq2$
\[
  \mathcal A(\mathcal LT_m)=m,
  \qquad
  \mathcal A(\mathcal L^2T_m)=\frac32m^3,
  \qquad
  \mathcal A(\mathcal L^3T_m)=\frac38m^3(5m^2+4).
\]
Putting $m=2k$ yields \eqref{eq:L-moments}--\eqref{eq:L3-moment}.

\begin{theorem}[Fixed-mode response asymptotics]\label{thm:Rnk-fixed-k}
For every fixed integer $k\geq1$, as $n\to\infty$,
\begin{equation}\label{eq:Rnk-fixed-expansion}
  \begin{aligned}
  R_{n,k}={}&\frac{k}{n^2}-\frac{k}{n^3}
  +\frac{3k(k^2+1)}{4n^4}
  -\frac{k(3k^2+1)}{2n^5}\\
  &+\frac{k(45k^4+123k^2+20)}{64n^6}
  +O_k(n^{-7}).
  \end{aligned}
\end{equation}
\end{theorem}

\begin{proof}
By \eqref{eq:Rnk-Legendre}, \eqref{eq:L-T-expansion}, and
\eqref{eq:Inj-omega},
\[
  R_{n,k}=\frac1{2n+1}\sum_{j=1}^{\min(n,k)}d_{k,j}I_{n,j}.
\]
Moreover,
\[
  \sum_{j=1}^kd_{k,j}\omega_j\lambda_j^r
  =\mathcal A(\mathcal L^{r+1}T_{2k}).
\]
For fixed $k$ and $n\geq k$, insert \eqref{eq:Inj-expansion}, use
\eqref{eq:L-moments}--\eqref{eq:L3-moment}, and expand the prefactor
$(2n+1)^{-1}$. Collecting equal powers of $1/n$ gives
\eqref{eq:Rnk-fixed-expansion}.
\end{proof}

For example, the first mode has the exact response
\begin{equation}\label{eq:Rn1-exact}
  R_{n,1}=\frac{16n(n+1)}{(2n-1)(2n+1)^2(2n+3)}.
\end{equation}
At the opposite end of the frequency range one has the following fixed-degree estimate.

\begin{proposition}[High-frequency response]\label{prop:Rnk-high-frequency}
For each fixed $n\geq1$,
\begin{equation}\label{eq:Rnk-high-frequency}
  R_{n,k}=-\frac{n(n+1)}{(2n+1)k^2}+O_n(k^{-4})
  \qquad(k\to\infty).
\end{equation}
\end{proposition}

\begin{proof}
Writing $x=\cos\theta$ in \eqref{eq:Rnk-Gn} gives
\[
  R_{n,k}=\int_0^\pi \mathcal H_n(\theta)\cos(2k\theta)\,\dd\theta,
  \qquad
  \mathcal H_n(\theta)=G_n(\cos\theta)\sin\theta.
\]
Here $\mathcal H_n(0)=\mathcal H_n(\pi)=0$ and
\[
  \mathcal H_n'(0)=\frac{2n(n+1)}{2n+1},
  \qquad
  \mathcal H_n'(\pi)=-\frac{2n(n+1)}{2n+1}.
\]
Two integrations by parts give the displayed main term; two further integrations by parts
give the stated remainder.
\end{proof}

For a fixed cutoff $K$, the expansion gives
\[
 \sum_{k=1}^K V_kR_{n,k}
 =\frac{S_K}{n^2}-\frac{S_K}{n^3}+O_K(n^{-4}),
 \qquad S_K=\sum_{k=1}^K kV_k.
\]
This identity is not uniform in $K$. The full sum is instead evaluated by
\eqref{eq:summed-mode-identity} and \cref{cor:full-background}.

To describe the divergent endpoint derivative, set
\[
 V_+(z)=\sum_{k\geq1}V_kz^k\quad(|z|<1),\qquad
 S_N=\sum_{k=1}^N kV_k.
\]
The sequence $S_N$ consists of the partial sums of the formal derivative $V_+'(1)$.

\begin{proposition}\label{prop:SN-digamma}
For every $N\geq1$,
\begin{equation}\label{eq:SN-digamma}
  S_N=\frac1\pi\int_0^\infty e^{-\pi t}
  \left(\operatorname{Re}\psi(N+1+it)
  -\operatorname{Re}\psi(1+it)\right)\,\dd t.
\end{equation}
\end{proposition}

\begin{proof}
Sum \eqref{eq:kVk-integral} and use
\[
  \sum_{k=1}^N\frac{k}{k^2+t^2}
  =\operatorname{Re}\sum_{k=1}^N\frac1{k+it}
  =\operatorname{Re}\bigl(\psi(N+1+it)-\psi(1+it)\bigr).
\]
\end{proof}

\begin{theorem}\label{thm:Sn-asymptotic}
As $N\to\infty$,
\begin{equation}\label{eq:SN-asymptotic}
  \begin{aligned}
  S_N={}&\frac{\log N}{\pi^2}+\mathfrak c
  +\frac{1}{2\pi^2N}\\
  &+\left(\frac1{\pi^4}-\frac1{12\pi^2}\right)\frac1{N^2}
  -\frac1{\pi^4N^3}+O(N^{-4}),
  \end{aligned}
\end{equation}
where
\begin{equation}\label{eq:c-constant}
  \mathfrak c
  =-\frac1\pi\int_0^\infty e^{-\pi t}
  \operatorname{Re}\psi(1+it)\,\dd t
\end{equation}
and equivalently
\begin{equation}\label{eq:c-series}
  \mathfrak c=\frac{\gamma}{\pi^2}
  +\sum_{k=1}^\infty\left(kV_k-\frac1{\pi^2k}\right).
\end{equation}
Numerical evaluation gives
\[
  \mathfrak c=0.0431127126701\ldots.
\]
\end{theorem}

\begin{proof}
Subtracting $\pi^{-2}\log N$ from \eqref{eq:SN-digamma} gives
\[
  S_N-\frac{\log N}{\pi^2}
  =\frac1\pi\int_0^\infty e^{-\pi t}
  \left(\operatorname{Re}\psi(N+1+it)-\log N
  -\operatorname{Re}\psi(1+it)\right)\,\dd t.
\]
The digamma expansion in the right half-plane \cite[Section 5.11(i)]{DLMF} gives,
for each fixed $t$,
$\operatorname{Re}\psi(N+1+it)-\log N\to0$. It also gives the uniform bound
\[
  \left|\operatorname{Re}\psi(N+1+it)-\log N\right|
  \ll 1+\log(1+t),
  \qquad N\geq1, t\geq0.
\]
The right-hand side is integrable against $e^{-\pi t}\,\dd t$, so dominated convergence
proves \eqref{eq:c-constant}.

For the lower-order terms, put
\[
  d_k=kV_k-\frac1{\pi^2k}.
\]
By \cref{thm:fourier-V},
\[
  d_k=-\frac2{\pi^4k^3}
  +\frac{24}{\pi^6k^5}+O(k^{-7}),
\]
so the following series converges absolutely. Set
\[
  \mathfrak c_*=\frac\gamma{\pi^2}+\sum_{k=1}^\infty d_k.
\]
Then
\[
  S_N=\frac{H_N}{\pi^2}+\mathfrak c_*-\frac\gamma{\pi^2}
  -\sum_{k>N}d_k,
\]
Euler--Maclaurin gives
\[
  H_N=\log N+\gamma+\frac1{2N}-\frac1{12N^2}+O(N^{-4}),
\]
\[
  \sum_{k>N}\frac1{k^3}
  =\frac1{2N^2}-\frac1{2N^3}+O(N^{-4}),
  \qquad
  \sum_{k>N}\frac1{k^5}=O(N^{-4}).
\]
Substitution gives \eqref{eq:SN-asymptotic} with $\mathfrak c_*$ in place of
$\mathfrak c$. Comparison of the constant term with the limit already obtained from
\eqref{eq:c-constant} shows that $\mathfrak c_*=\mathfrak c$, proving both
\eqref{eq:SN-asymptotic} and the equivalent series formula \eqref{eq:c-series}.
\end{proof}

The constant $\mathfrak c$ belongs to the truncated endpoint sum, whereas
$C_{\mathrm{lin}}$ belongs to the full first variation. These are distinct limits;
neither determines the conjectural constant $C_\alpha$.

\section*{Statements and Declarations}\label{sec:declarations}

\noindent\textbf{Funding.}
This work was supported by the National Natural Science Foundation of China
(Grant Nos.~12231009 and 11971224).

\medskip
\noindent\textbf{Competing interests.}
The author has no relevant financial or non-financial interests to disclose.

\medskip
\noindent\textbf{Data availability.}
The results are proved analytically in the article and do not rely on an external dataset.
The reported numerical constants are determined by the explicit convergent integrals and
series given in the text.

\medskip
\noindent\textbf{Use of generative AI.}
ChatGPT (OpenAI) was used to assist with language editing.
The author takes full responsibility for the mathematical content, computations,
references, and final wording of the manuscript.

\end{document}